\documentclass[11pt,twoside]{amsart}
\usepackage{amsxtra}
\usepackage{color}
\usepackage{amsopn}
\usepackage{amsmath,amsthm,amssymb,arydshln}
\usepackage{mathrsfs,mathtools}
\usepackage{stmaryrd}
\usepackage{hyperref}
\usepackage{enumerate}
\usepackage{xcolor}
\usepackage{pifont}
\usepackage{adjustbox}
\usepackage{tikz}

\newtheorem{theorem}{Theorem}[section]
\newtheorem{corollary}[theorem]{Corollary}

\newtheorem*{theorem*}{Theorem}
\theoremstyle{definition}

\newtheorem{example}[theorem]{Example}
\newtheorem{remark}[theorem]{Remark}

\newcommand{\R}{{\mathbb R}}

\newcommand{\Z}{\mathbb Z}
\newcommand{\T}{{\mathbb T}}

\newcommand{\beq}{\begin{equation}}
\newcommand{\eeq}{\end{equation}}

\newcommand{\f}{\varphi}

\newcommand{\psip}{\psi}

\newcommand{\SU}{{\mathrm{SU}}}

\newcommand{\G}{{\mathrm G}}

\newcommand{\W}{\wedge}
\newcommand{\Diff}{\mathrm{Dif{}f}}
\newcommand{\Iso}{\mathrm{Iso}}

\DeclareMathOperator\Aut{Aut}

\newcommand{\frn}{\mathfrak{n}}

\renewcommand{\gg}{\mathfrak{g}}

\newcommand{\gn}{\mathfrak{n}}

\newcommand{\st}{\ |\ }

\newcommand{\sst}{\scriptscriptstyle}

\numberwithin{equation}{section}

\title[On the automorphism group of a closed G$_2$-structure]{On the automorphism group of a closed G$_{\mathbf2}$-structure}
\author{Fabio Podest\`a and Alberto Raffero}
\thanks{The authors were supported by GNSAGA of INdAM}
\address{Dipartimento di Matematica e Informatica ``U.~Dini'' \\ Universit\`a degli Studi di Firenze\\ Viale Morgagni 67/a\\ 50134 Firenze\\ Italy}
\email{podesta@math.unifi.it, alberto.raffero@unifi.it}

\begin{document}
\begin{abstract}
We study the automorphism group of a compact 7-manifold $M$ endowed with a closed non-parallel $\G_2$-structure, showing that its identity component is abelian 
with dimension bounded by $\mathrm{min}\{6,b_2(M)\}$. 
This implies the non-existence of compact homogeneous manifolds endowed with an invariant closed non-parallel $\G_2$-structure. 
We also discuss some relevant examples. 
\end{abstract}

\maketitle

\section{Introduction}
A seven-dimensional smooth manifold $M$ admits a $\G_2$-structure if the structure group of its frame bundle can be reduced to 
the exceptional Lie group $\G_2\subset\mathrm{SO}(7)$.  
Such a reduction is characterized by the existence of a global 3-form $\f\in\Omega^3(M)$ satisfying a suitable non-degeneracy condition and 
giving rise to a Riemannian metric $g_\f$ and to a volume form $dV_\f$ on $M$ via the identity
\[
g_\f(X,Y)\,dV_\f = \frac{1}{6}\,\iota_X\f\W\iota_Y\f\W\f,
\]
for all $X,Y\in\mathfrak{X}(M)$ (see e.g.~\cite{Bry,Hit}). 

By \cite{FeGr}, the intrinsic torsion of a $\G_2$-structure $\f$ can be identified with the covariant derivative $\nabla^{g_\f}\f$, and it vanishes identically if and only if 
both $d\f=0$ and $d*_\f\f=0$, $*_\f$ being the Hodge operator defined by $g_\f$ and $dV_\f$. 
On a compact manifold, this last fact is equivalent to $\Delta_\f\f=0$, where $\Delta_\f = d^*d+dd^*$ is the Hodge Laplacian of $g_\f$.  
A $\G_2$-structure $\f$ satisfying any of these conditions is said to be {\em parallel} and its associated Riemannian metric $g_\f$ has holonomy contained in $\G_2$.  
Consequently, $g_\f$ is Ricci-flat and the automorphism group $\Aut(M,\f) \coloneqq \left\{f\in\Diff(M)\st f^*\f=f \right\}$ of $(M,\f)$ is finite when $M$ is compact and 
$\mathrm{Hol}(g_\f)=\G_2$. 
 
Parallel $\G_2$-structures play a central role in the construction of compact manifolds with holonomy $\G_2$, 
and known methods to achieve this result involve {\em closed} $\G_2$-structures, i.e., those whose defining 3-form $\f$ satisfies $d\f=0$ 
(see \cite{Bry,BrXu,CHNP,Joy,Kov,LoWe}).   

Most of the known examples of 7-manifolds admitting closed $\G_2$-structures consist of simply connected Lie groups endowed with a left-invariant closed $\G_2$-form 
$\f$ \cite{CoFe,Fer,Fer1,FiRa,Lau}. Compact locally homogeneous examples can be obtained considering the quotient of such groups by a 
co-compact discrete subgroup, whenever this exists. 
Further non-homogeneous closed $\G_2$-structures on the 7-torus can be constructed starting from the symplectic half-flat $\SU(3)$-structure on $\mathbb{T}^6$ 
described in \cite[Ex.~5.1]{DeTo} (see Example \ref{ExTorus} for details). 

Up to now, the existence of compact homogeneous 7-manifolds admitting an invariant closed non-parallel $\G_2$-structure is still not known 
(cf.~\cite[Question 3.1]{Lau} and \cite{LeMu,Rei}). 
Moreover, among the $\G_2$-manifolds acted on by a cohomogeneity one simple group of automorphisms studied in \cite{ClSw} no compact examples admitting a closed 
$\G_2$-structure occur.  

In this short note, we investigate the properties of the automorphism group $\Aut(M,\f)$ of a compact 7-manifold $M$ endowed with a closed non-parallel $\G_2$-structure $\f$. 
Our main results are contained in Theorem \ref{MainThm}, where we show that the identity component $\Aut(M,\f)^{\sst0}$ is necessarily abelian 
with dimension bounded by $\mathrm{min}\{6,b_2(M)\}$. 
In particular, this answers negatively \cite[Question 3.1]{Lau} and explains why compact examples cannot occur in \cite{ClSw}. 
Moreover, we also prove some interesting properties of the automorphism group action. 
Finally, we describe some relevant examples.  

Similar results hold for compact symplectic half-flat 6-manifolds, and they will appear in a forthcoming paper. 

\section{The automorphism group}
Let $M$ be a seven-dimensional manifold endowed with a closed $\G_2$-structure $\f$, and consider its automorphism group
\[
\Aut(M,\f) \coloneqq \left\{f\in\Diff(M)\st f^*\f=f \right\}.
\]
Notice that $\Aut(M,\f)$ is a closed Lie subgroup of $\Iso(M,g_\f)$, and that the Lie algebra of its identity component $\G \coloneqq \Aut(M,\f)^{\sst0}$ is 
\[
\gg=\left\{X\in\mathfrak{X}(M) \st \mathcal{L}_X\f=0 \right\}.
\]
In particular, every $X\in\gg$ is a Killing vector field for the metric $g_\f$ (cf.~\cite[Lemma 9.3]{LoWe}). 

When $M$ is compact,  the Lie group $\Aut(M,\f)\subset\Iso(M,g_\f)$ is also compact, and we can show the following. 
\begin{theorem}\label{MainThm}
Let $M$ be a compact seven-dimensional manifold endowed with a closed non-parallel $\G_2$-structure $\f$. 
Then, there exists an injective map
\[
F:\gg\rightarrow\mathscr{H}^2(M),\quad X\mapsto \iota_X\f,
\]
where $\mathscr{H}^2(M)$ is the space of $\Delta_\f$-harmonic 2-forms.  
As a consequence, the following properties hold:
\begin{enumerate}[1)]
\item\label{thm1} $\dim(\gg)\leq b_2(M)$;
\item\label{thm2} $\gg$ is abelian with $\dim(\gg)\leq6$;
\item\label{thm3} for every $p\in M,$ the isotropy subalgebra $\gg_p$ has dimension $\dim(\gg_p)\leq2$, with equality only when $\dim(\gg)=2,3$;
\item\label{thm4} the $\G$-action is free when $\dim(\gg)\geq5$. 
\end{enumerate}
\end{theorem}
\begin{proof}
Let $X\in\gg$. Then, $0=\mathcal{L}_X\f = d(\iota_X\f)$, as $\f$ is closed. We claim that $\iota_X\f$ is co-closed (see also \cite[Lemma 9.3]{LoWe}). 
Indeed, by \cite[Prop.~A.3]{Kar} we have
\[
\iota_X\f\W\f = -2*_\f(\iota_X\f),
\]
from which it follows that 
\[
0 = d(\iota_X\f\W\f) =  -2d*_\f(\iota_X\f).
\]
Consequently, the 2-form $\iota_X\f$ is $\Delta_\f$-harmonic and $F$ is the restriction of the injective map $Z\mapsto\iota_Z\f$ 
to $\gg$. From this \ref{thm1}) follows. 

As for \ref{thm2}), we begin observing that $\mathcal{L}_Y(\iota_X\f)=0$ for all $X,Y\in\gg$, since every Killing field on a compact manifold preserves every harmonic form.  
Hence, we have
\[
0 = \mathcal{L}_Y(\iota_X\f) = \iota_{[Y,X]}\f + \iota_X(\mathcal{L}_Y\f) =  \iota_{[Y,X]}\f.
\]
This proves that $\gg$ is abelian, the map $Z\mapsto \iota_Z\f$ being injective. 
Now, $\G$ is compact abelian and it acts effectively on the compact manifold $M.$ Therefore, the principal isotropy is trivial and $\dim(\gg)\leq7$.  
When $\dim(\gg) = 7$, $M$ can be identified with the 7-torus $\mathbb{T}^7$ endowed with a left-invariant metric, which is automatically flat. 
Hence, if $\f$ is closed non-parallel, then $\dim(\gg)\leq6$. 

In order to prove \ref{thm3}), we fix a point $p$ of $M$ and we observe that the image of the isotropy representation $\rho:\G_p\rightarrow \mathrm{O}(7)$
is conjugated into $\G_2$. Since $\G_2$ has rank two and $\G_p$ is abelian, the dimension of $\gg_p$ is at most two. 
If $\dim(\gg_p)=2$, then the image of $\rho$ is conjugate to a maximal torus of $\G_2$ and its fixed point set in $T_pM$ is one-dimensional. 
As $T_p(\G\cdot p)\subseteq (T_pM)^{\G_p}$, the dimension of the orbit $\G\cdot p$ is at most one, which implies that $\dim(\gg)$ is either two or three. 

The last assertion \ref{thm4}) is equivalent to proving that $\G_p$ is trivial for every $p\in M$ whenever $\dim(\gg)\geq5$. 
In this case, $\dim(\gg_p)\leq1$ by \ref{thm3}). Assuming that the dimension is precisely one, then the dimension of the orbit $\G\cdot p$ is at least four. 
This means that the $\G_p$-fixed point set in $T_pM$ is at least four-dimensional. 
On the other hand, the fixed point set of a closed one-parameter subgroup of $\G_2$ is at most three-dimensional. This gives a contradiction. 
\end{proof}

The following corollary answers negatively a question posed by Lauret in \cite{Lau}. 
\begin{corollary}
There are no compact homogeneous 7-manifolds endowed with an invariant closed non-parallel $\G_2$-structure. 
\end{corollary}
\begin{proof}
The assertion follows immediately from point \ref{thm2}) of Theorem \ref{MainThm}. 
\end{proof}

In contrast to the last result, it is possible to exhibit non-compact homogeneous examples. 
Consider for instance a six-dimensional non-compact homogeneous space $\mathrm{H}/\mathrm{K}$ endowed with an invariant symplectic half-flat $\SU(3)$-structure, 
namely an $\SU(3)$-structure $(\omega,\psip)$ such that $d\omega=0$ and $d\psip=0$ 
(see \cite{PoRa} for the classification of such spaces when $\mathrm{H}$ is semisimple and for more information on symplectic half-flat structures). 
If $(\omega,\psip)$ is not torsion-free, i.e., if $d(J\psip)\neq0$, 
then the non-compact homogeneous space $(\mathrm{H}\times\mathbb{S}^1)/\mathrm{K}$ admits an invariant closed non-parallel $\G_2$-structure defined by the 3-form 
\[
\f \coloneqq \omega\W ds+\psip,
\]
where $ds$ denotes the global 1-form on $\mathbb{S}^1$. 

\begin{remark}
In \cite{ClSw}, the authors investigated $\G_2$-manifolds acted on by a cohomogeneity one simple group of automorphisms. 
Theorem \ref{MainThm} explains why compact examples in the case of closed non-parallel $\G_2$-structures do not occur. 
\end{remark}

The next example shows that $\G$ can be non-trivial, that the upper bound on its dimension given in \ref{thm2}) can be attained, 
and that \ref{thm4}) is only a sufficient condition. 
\begin{example}\label{ExTorus}
In \cite{DeTo}, the authors constructed a symplectic half-flat $\SU(3)$-structure $(\omega,\psi)$ on the 6-torus $\mathbb{T}^6$ as follows. 
Let $(x^1,\ldots,x^6)$ be the standard coordinates on $\R^6$, and let $a(x^1)$, $b(x^2)$ and $c(x^3)$ be three smooth functions on $\R^6$ such that 
\[
\lambda_1 \coloneqq b(x^2)-c(x^3),\quad \lambda_2 \coloneqq c(x^3)-a(x^1),\quad \lambda_3 \coloneqq a(x^1)-b(x^2),
\]
are $\Z^6$-periodic and non-constant. 
Then, the following pair of $\Z^6$-invariant differential forms on $\R^6$ induces an $\SU(3)$-structure on $\T^6 = \R^6/\Z^6$:
\begin{eqnarray*}
\omega 	&=& dx^{14}+dx^{25}+dx^{36},\\
\psip 	&=& -e^{\lambda_3}\,dx^{126} +e^{\lambda_2}\,dx^{135}  -e^{\lambda_1}\,dx^{234} +dx^{456},
\end{eqnarray*}
where $dx^{ijk\cdots}$ is a shorthand for the wedge product $dx^i\W dx^j \W dx^k \W \cdots$. 
It is immediate to check that both $\omega$ and $\psip$ are closed and that $d(J\psip)\neq0$ whenever at least one of the functions $a(x^1)$, $b(x^2)$, $c(x^3)$ is not identically zero.  
Thus, the pair $(\omega,\psip)$ defines a symplectic half-flat $\SU(3)$-structure on the 6-torus. 
The automorphism group of $(\mathbb{T}^6,\omega,\psip)$ is $\mathbb{T}^3$ when $a(x^1)\,b(x^2)\, c(x^3) \not\equiv0$, 
while it becomes $\mathbb{T}^4$  ($\mathbb{T}^5$) when one (two) of them vanishes identically.  

Now, we can consider the closed $\G_2$-structure on $\mathbb{T}^7=\mathbb{T}^6\times\mathbb{S}^1$ defined by the 3-form $\f = \omega\W ds +\psi$. 
Depending on the vanishing of none, one or two of the functions $a(x^1)$, $b(x^2)$, $c(x^3)$, $\f$ is a closed non-parallel $\G_2$-structure and 
the automorphism group of $(\mathbb{T}^7,\f)$ is $\mathbb{T}^{4}$, $\mathbb{T}^{5}$ or $\mathbb{T}^{6}$, respectively. 
\end{example}

Finally, we observe that there exist examples where the upper bound on the dimension of $\gg$ given in \ref{thm1}) is more restrictive than the upper bound given in \ref{thm2}). 

\begin{example}
In \cite{CoFe}, the authors obtained the classification of seven-dimensional nilpotent Lie algebras admitting closed $\G_2$-structures. 
An inspection of all possible cases shows that the Lie algebras whose second Betti number is lower than seven are those appearing in Table \ref{closedG2Lieb2}. 
\begin{table}[ht]
\centering
\renewcommand\arraystretch{1.4}
\begin{tabular}{|c|c|}
\hline
nilpotent Lie algebra $\gn$													& 	$b_2(\gn)$	\\ \hline \hline
$(0, 0, e^{12}, e^{13}, e^{23}, e^{15} + e^{24}, e^{16} + e^{34})$						&	$3$			\\ \hline 
$(0, 0, e^{12}, e^{13}, e^{23}, e^{15} + e^{24}, e^{16} + e^{34} + e^{25})$					&	$3$			\\ \hline    
$(0, 0, e^{12}, 0, e^{13} + e^{24}, e^{14}, e^{46} + e^{34} + e^{15} + e^{23})$				&	$5$			\\ \hline    
$(0, 0, e^{12}, 0, e^{13}, e^{24} + e^{23}, e^{25} + e^{34} + e^{15} + e^{16} - 3 e^{26})$		&	$6$			\\ \hline    
\end{tabular}
\vspace{0.1cm}
\caption{}\label{closedG2Lieb2}
\end{table}
\renewcommand\arraystretch{1}

Let $\gn$ be one of the Lie algebras in Table \ref{closedG2Lieb2}, and consider a closed non-parallel $\G_2$-structure $\f$ on it. 
Then, left multiplication extends $\f$ to a left-invariant $\G_2$-structure of the same type on the simply connected nilpotent Lie group $\mathrm{N}$ corresponding to $\frn$. 
Moreover, as the structure constants of $\frn$ are integers, there exists a co-compact discrete subgroup $\Gamma\subset \mathrm{N}$ giving rise to a compact nilmanifold 
$\Gamma\backslash \mathrm{N}$ \cite{Mal}. 
The left-invariant 3-form $\f$ on $\mathrm{N}$ passes to the quotient defining an invariant closed non-parallel $\G_2$-structure on $\Gamma\backslash \mathrm{N}$. 
By Nomizu Theorem \cite{Nom}, the de Rham cohomology group $\mathrm{H}^k_{\mathrm{dR}}(\Gamma\backslash \mathrm{N})$ is isomorphic to the 
cohomology group $\mathrm{H}^k(\gn^*)$ of the Chevalley-Eilenberg complex of $\frn$. Hence, $b_2(\Gamma\backslash \mathrm{N})=b_2(\gn)$.
\end{example}

\end{document}